\documentclass[12pt,a4paper]{article}

\usepackage{latexsym}
\usepackage{amsfonts}
\usepackage{amsmath}
\usepackage{amsthm}
\usepackage{dsfont}
\usepackage{bbm}

\author{K\'aroly J. B\"or\"oczky
\thanks{Alfr\'ed R\'enyi Institute of Mathematics, Budapest, Hungary, boroczky.karoly.j@renyi.hu}
\thanks{supported by Hungarian Grant NKFIH 132002}
\and 
G\'abor Lugosi\thanks{Department of Economics and Business, Pompeu
  Fabra University, Barcelona, Spain, gabor.lugosi@upf.edu}
\thanks{ICREA, Pg. Lluís Companys 23, 08010 Barcelona, Spain}
\thanks{Barcelona Graduate School of Economics}
\thanks{
supported by
the Spanish Ministry of Economy and Competitiveness,
Grant MTM2015-67304-P and FEDER, EU. 
}
\and 
Matthias Reitzner
\thanks{Institut f\"ur Mathematik,
Universit\"at Osnabr\"uck, Germany, matthias.reitzner@uni-osnabrueck.de}
\thanks{supported by the German Research Foundation DFG-GRK 1916}
}
\title{Facets of high-dimensional Gaussian polytopes}

\def\1{{\mathds 1}}

\def\a{{\alpha}}
\def\b{{\beta}}
\def\d{{\delta}}

\def\l{{\lambda}}
\def\p{{\phi}}
\DeclareMathOperator\lln{{lln}}

\newcommand{\E}{\mathbb{E}}
\def\N{\mathbb{N}}
\def\P{\mathbb{P}}
\newcommand{\R}{\mathbb{R}}
\newcommand{\Rd}{\R^d}

\def\bB{\boldsymbol{B}}

\newcommand{\var}{\mathrm{var}}

\begin{document}

\maketitle

\newtheorem{theo}{Theorem}[section]
\newtheorem{coro}[theo]{Corollary}
\newtheorem{lemma}[theo]{Lemma}
\newtheorem{remark}[theo]{Remark}
\newtheorem{prop}[theo]{Proposition}
\newtheorem{conj}[theo]{Conjecture}
\newtheorem{example}[theo]{Example}

\begin{abstract}
We study the number of facets of the convex hull of $n$ independent standard Gaussian points in $\Rd$.
In particular, we are interested in the expected number of facets when the dimension is allowed to grow with the 
sample size. We establish an explicit asymptotic formula that is valid whenever $d/n\to 0$. We also obtain
the asymptotic value when $d$ is close to $n$.
\end{abstract}

\section{Introduction}
\label{secmain}

The convex hull $[X_1,\ldots,X_n]$ of $n$ independent standard Gaussian samples $X_1,\ldots,X_n$ from $\R^d$ is the Gaussian polytope $P^{(d)}_n$. For fixed dimension $d$, the face numbers and intrinsic volumes of $P_n^{(d)}$ as $n$ tends to infinity are well understood by now. For $i=0\ldots,d$ and polytope $Q$, let $f_i(Q)$ denote the number of $i$-faces of $Q$ and let $V_i(Q)$ denote the $i$th intrinsic volume of
$Q$. The asymptotic behavior of the expected value of the number of facets $f_{d-1}(P^{(d)}_n)$ as $n \to \infty$ was provided by R\'enyi, Sulanke \cite{ReS63} if $d=2$, and by Raynaud \cite{Ray70} if $d\geq 3$. Namely,  they proved that, for any fixed $d$,
\begin{equation}
\label{eq:facetnlargefixd}
\E f_{d-1}(P^{(d)}_n)
=
2^d \pi^{\frac {d-1}2} d^{- \frac 12} 
( \ln  n)^{ \frac {d-1}2 } (1+o(1))
\end{equation}
as $n \to \infty$. For $i=0,\ldots,d$, expected value of $V_i(P^{(d)}_n)$ as $n \to \infty$ was computed by 
Affentranger \cite{Aff91}, and that of $f_i(P^{(d)}_n)$ was determined Affentranger, Schneider \cite{AfS92} and Baryshnikov, Vitale \cite{BaV94}, see  Hug, Munsonius, Reitzner \cite{HMR04} and Fleury \cite{Fle12} for a different approach. More recently, Kabluchko and Zaporozhets \cite{KaZa19, KaZa20} proved explicit expressions for the expected value of $V_d(P^{(d)}_n)$ and the number of $k$-faces $f_k(P^{(d)}_n)$. Yet these formulas are complicated and it is not immediate how to deduce asymptotic results for large $n$ high dimensions $d$.

After various partial results, including the variance estimates of Calka, Yukich \cite{CaY15} and Hug, Reitzner \cite{HuR05},
central limit theorems were proved for  $f_i(P^{(d)}_n)$ and $V_d(P^{(d)}_n)$ by B\'ar\'any and Vu \cite{BaV07}, and for $V_i(P^{(d)}_n)$ by B\'ar\'any and Th\"ale \cite{BaT17}. These results have been strengthened considerably by Grote and Th\"ale \cite{GrTh18}.
The interesting question whether $\E f_{d-1}(P^{(d)}_n)$ is an increasing function in $n$ was answered in the positive by Kabluchko and Th\"ale \cite{KaTh18}. It would be interesting to investigate the monotonicity behavior of the facet number if $n$ and $d$ increases simultaneously.

The ``high-dimensional'' regime, that is, when $d$ is allowed to grow with $n$, is of interest in numerous
applications in statistics, signal processing, and information theory. 
The combinatorial structure of $P^{(d)}_n$, when $d$ tends to infinity and $n$ grows proportionally with $d$, 
was first investigated by Vershik and Sporyshev \cite{VeS92},  and later Donoho and Tanner \cite{DoT09} provided a satisfactory description. For any $t>1$, Donoho, Tanner \cite{DoT09} determined the optimal $\varrho(t)\in(0,1)$ such that if $n/d$ tends to $t$, then $P^{(d)}_n$ is essentially $\varrho(t)d$-neighbourly 
(if $0<\eta<\varrho(t)$ and $0\leq k\leq \eta d$, then $f_k(P^{(d)}_n)$ is asymptotically ${n \choose k+1}$). See
Donoho \cite{Don06}, Cand\'es, Romberg, and Tao \cite{CRT06}, Cand\'es and Tao \cite{CaT05,CaT06},
 Mendoza-Smith, Tanner, and Wechsung \cite{MTW18+}.

In this note, we consider $f_{d-1}(P^{(d)}_n)$, the number of facets, when both $d$ and $n$ tend to infinity. 
Our main result is the following estimate for the expected number of facets of the Gaussian polytope. 
The implied constant in $O(\cdot)$ is always some absolute constant. 
We write $\lln x$ for $\ln(\ln x)$.

\begin{theo}
\label{th:facetnlarge}
Assume $P^{(d)}_n$ is a Gaussian polytope. Then for $d \geq 78$  and $n \geq e^e d$, we have
$$
\E f_{d-1}(P^{(d)}_n)
=
2^d \pi^{\frac {d-1}2} d^{- \frac 12} 
e^{ \frac {d-1}2 \lln  \frac nd   
- \frac{d-1}4 \frac {\lln \frac nd }{\ln \frac nd  }
+(d-1) \frac {\theta}{\ln \frac nd}  
\  
+O( \sqrt d e^{ - \frac 1{10} d})} 
$$
with $\theta=\theta(n,d) \in [-34,2]$.
\end{theo}
When $n/d$ tends to infinity as $d \to \infty$, Theorem \ref{th:facetnlarge} provides the asymptotic formula
$$
\E f_{d-1}(P^{(d)}_n)
=
\left((4\pi+o(1)) \ln \frac nd\right)^{\frac {d-1}2}~.
$$
If $n/(d e^d)\to \infty$, then we have 
$\frac{d} {\ln \frac nd} \to 0 $
and hence
$$
\E f_{d-1}(P^{(d)}_n)
=
2^d \pi^{\frac {d-1}2} d^{- \frac 12} 
e^{ \frac {d-1}2 \lln  \frac nd   
- \frac{d-1}4 \frac {\lln \frac nd }{\ln \frac nd  }
+o(1)} 
$$
as $d \to \infty$.
In the case when $n$ grows even faster such that $(\ln n)/(d\ln d) \to \infty$,
the asymptotic formula simplifies to the result 
(\ref{eq:facetnlargefixd}) of R\'enyi, Sulanke \cite{ReS63} and Raynaud \cite{Ray70} for fixed 
dimension.

\begin{coro}
\label{facetnlargecor}
Assume $P^{(d)}_n$ is a Gaussian polytope.  If  $(\ln n)/(d\ln d) \to \infty$, we have 
$$
\E f_{d-1}(P^{(d)}_n)
=
2^d \pi^{\frac {d-1}2} d^{- \frac 12} 
( \ln  n)^{ \frac {d-1}2 } (1+o(1))~.
$$
\end{coro}

There is a (simpler) counterpart of our main results stating the asymptotic behavior of the 
expected number of facets of $P_n^{(d)}$, if 
$n-d$ is \emph{small} compared to $d$, that is, if $n/d$ tends to one. 
\begin{theo}
\label{th:facetn-dsmall}
Assume $P^{(d)}_n$ is a Gaussian polytope. Then
for $n - d = o( d)$, we have
$$
\E f_{d-1}(P^{(d)}_n)
=
{n \choose d}
2^{-(n-d)+1} e^{\frac 1\pi \frac{(n-d)^2}d + O\left(\frac{(n-d)^3}{d^2}\right) +o(1) }
$$
as $d \to \infty$.
\end{theo}
This complements a result of Affentranger and Schneider \cite{AfS92} stating the number of 
$k$-dimensional faces for $k \leq n-d$ and $n-d$ fixed,
$$
\E f_{k}(P^{(d)}_n)
=
{n \choose k+1}
(1 +o(1))~,
$$
as $d \to \infty$.

In the next section we sketch the basic idea of our approach, leaving the technical details to later 
sections. In Section \ref{sec:phi} we provide asymptotic approximations
for the tail of the normal distribution. 
In Section \ref{sec:beta} concentration inequalities are derived for the $\beta$-distribution.
Finally, in Sections \ref{sec:onedim} and \ref{sec:small}, Corollary \ref{facetnlargecor} and Theorem \ref{th:facetn-dsmall}
are proven.

\section{Outline of the argument}
\label{secidea}

For $z\in\R$, let
$$
\Phi(y)=\frac1{\sqrt{\pi}}\int\limits_{-\infty}^y e^{- s^2}\,ds,
\ \mbox {and } 
\p(y)=\Phi'(y)=\frac1{\sqrt{\pi}}e^{- y^2}~.
$$
Our proof is based on the approach of Hug, Munsonius, and Reitzner \cite{HMR04}. 
In particular, \cite[Theorem~3.2]{HMR04} states
that if $n\geq d+1$ and $X_1,\ldots,X_n$
are independent standard Gaussian points in $\R^d$, then
$$
\E f_{d-1}([X_1,\ldots,X_n])
=
{n \choose d} \P(Y \notin [Y_1, \dots, Y_{n-d}])~,
$$
where $Y,Y_{1},\dots,Y_{n-d}$ are independent real-valued random variables 
with $Y \stackrel{d}{=} N\left(0,\frac{1}{2d}\right)$ and  $Y_{i} \stackrel{d}{=} N\left(0,\frac{1}{2}\right)$ for $i=1,\ldots,n-d$.
This gives
\begin{eqnarray}\label{eq:main-integral}
\E f_{d-1}([X_1,\ldots,X_n])
&=&
2 {n \choose d}
\frac{\sqrt{d}}{\sqrt \pi}
\int\limits_{- \infty}^\infty \Phi(y)^{n-d} e^{-dy^2}\,dy
\\ &=&
2 {n \choose d}
\sqrt{d} \, \pi^{\frac {d-1}2}
\int\limits_{- \infty}^\infty \Phi(y)^{n-d} \p(y)^d \,dy~.
\end{eqnarray}
Note that similar integrals appear in the analysis of the expected number of $k$-faces for values of $k$ in the entire range
$k=0,\ldots,d-1$.
In our case, the analysis boils down to understanding the integral of
$\Phi(y)^{n-d} \p(y)^d$ over the real line.
By substituting
$ (1-u) = \Phi(y)$,
we obtain
$$
\int\limits_{- \infty}^\infty \Phi(y)^{n-d} \p(y)^d \,dy 
=
\int\limits_{0}^1 (1-u)^{n-d} \p(\Phi^{-1}(1-u))^{d-1}\,du~.
$$
Clearly, $n \geq d+2$ is the nontrivial range.
When
$n/d \to \infty $, $(1-u)^{n-d}$ is dominating, and we need to investigate 
the asymptotic behavior of $\p(\Phi^{-1}(1-u))$ as $u \to 0$. We show that the essential term 
is precisely $2u$. 
Hence, it makes sense to rewrite the integral as
$$
2^{d-1} \int\limits_{0}^1 (1-u)^{n-d} u^{d-1} \underbrace{\left( (2u)^{-1} \p(\Phi^{-1}(1-u))\right)^{d-1}}
_{=: g_d(u)} \,du~. 
$$
For $x,y>0$, the Beta-function is given by $\bB (x,y)= \int\limits_0^1 (1-u)^{x-1} u^{y-1} du $. 
It is well known that for $k,l \in \N$ we have 
$\bB (k,l)= 
\frac{(k-1)!(l-1)!}{(k+l-1)!}$. A random variable $U$ is $\bB_{(x,y)}$ distributed if its 
density is given by $\bB(x,y)^{-1} (1-u)^{x-1} u^{y-1}$.
With this, we have established the following identity:

\begin{prop}
\begin{eqnarray}\label{eq:exp-f-exp-U}
\E f_{d-1}([X_1,\ldots,X_n])
&=&
2^d \pi^{\frac {d-1}2} d^{- \frac 12} 
\E g_d(U) 
\end{eqnarray}
where 
$$g_d(u) = \left( (2u)^{-1} \p(\Phi^{-1}(1-u))\right)^{d-1}$$
and $U$ is a $\bB(n-d+1,d)$ random variable.
\end{prop}

In Lemma \ref{le:phi-Phi-inv} below we show that
$$
 g_d(u) 
=
(\ln u^{-1})^{-  \frac {d-1}2 }
e^{- \frac {d-1}4 \frac{\lln u^{-1} }{ \ln u^{-1}} 
- (d-1) \frac{O(1)}{\ln  u^{-1}} }
$$
as $u \to 0$.
Because the Beta function is concentrated around $\frac dn$, see Lemma \ref{le:betalow-} and Lemma \ref{le:betalarge}, this yields
$$ \E g_d(U) \approx 
\left( \ln \frac nd\right)^{\frac {d-1}2} e^{ - \frac{d-1}4 \frac{\lln \frac nd }{ \ln \frac nd} 
- (d-1) \frac{O(1)}{ \ln \frac nd} }
$$
which implies our main result.

\section{Asymptotics of the $\Phi$-function}
\label{sec:phi}

To estimate $\Phi(z)$, we need a version of
Gordon's inequality \cite{Gor41} for the Mill's ratio:

\begin{lemma}
\label{le:Phi-asymp}
For any $z>1$ there exists $\theta\in(0,1)$, such that
$$
\Phi(z)=
1-\frac{e^{-z^2}}{2\sqrt{\pi}z}
\left(1 - \frac{\theta}{2z^2}\right)
$$
\end{lemma}

\begin{proof}
It follows by partial integration that
$$
\int\limits_z^{\infty}e^{-t^2}\,dt=
\int\limits_z^{\infty}2t e^{-t^2}\, \frac {1}{2t} \,dt=
 \frac {e^{-z^2}}{2z}  - 
\int\limits_z^{\infty} \frac {e^{-t^2}}{2t^2} \,dt=
\frac {e^{-z^2}}{2z}  - 
\frac {\theta e^{-z^2}}{4z^3} 
$$
which yields the lemma. 
\end{proof}

\begin{lemma}
For any $u \in (0,e^{-1}]$ there is a $\d=\d(u) \in (0,16)$ such that 
\begin{equation}\label{eq:Phiinv-asymp}
  \Phi^{-1}(1-u) = 
\sqrt{
\ln u^{-1} - \frac 12 \lln u^{-1} - \ln (2\sqrt \pi) + \frac 14 \frac{\lln u^{-1} }{ \ln u^{-1}} 
+ \frac{\delta}{\ln  u^{-1}} 
}   .
\end{equation}
\end{lemma}
\begin{proof}
It is useful to prove (\ref{eq:Phiinv-asymp}) for the transformed variable $u=e^{-t}$.
We define 
\begin{equation}\label{eq:def-z}
z (t)
=
\sqrt{t - \frac 12 \ln t - \ln (2\sqrt \pi) + \frac 14 \frac{\ln t}{t} +\frac{\delta(t)}{t} } 
\end{equation}
which exists for $t > 0$.
In a first step we prove that this is the asymptotic expansion of $z=\Phi^{-1}(1-e^{-t})$ as $z,t \to \infty$ with a suitable function $\d=\d(t)=O(1)$. In a second step we show the bound on $\delta$. Observe that $z \geq 1$ implies $t \geq \ln \Phi(-1))=-2,54\dots$.
By Lemma \ref{le:Phi-asymp}, for $ z \geq 1$
\begin{equation}\label{eq:et-Phi-eq}
e^{-t}=
1- \Phi(z)
= \frac 1{2\sqrt \pi\, z} e^{- z^2} \left(1 - \frac {\theta(z)}{2 z^2}\right) 
\end{equation}
as $z \to \infty$ with some $\theta(z) \in (0,1)$, which immediately implies that 
$ z=z(t) \to \infty $ as $ t \to \infty $.
Equation (\ref{eq:et-Phi-eq}) shows that 
$e^t \geq 2 \sqrt \pi z e^{z^2}$ and thus 
$$ t \geq \ln(2 \sqrt \pi) + \ln  z(t) + z(t)^2 \geq z(t)^2 $$
for $z \geq 1$.
The function $z=z(t)$ is the inverse function we are looking for, if it satisfies 
\begin{equation}\label{eq:et-Phi-eq2}
4 \pi z(t)^2 e^{-2t} =  e^{- 2z(t)^2} \left(1 - \frac {\theta(z)}{2 z^2}\right) ^2.
\end{equation}
We plug (\ref{eq:def-z}) into this equation. This leads to
\begin{align*}
t - \frac 12 \ln t - \ln (2\sqrt \pi) + \frac 14 \frac{\ln t }{t} + \frac{\delta(t)}{ t} 
&=
t e^{ - \frac 12 \frac{\ln t}{t}  - 2\frac{\delta(t)}{t}  }
\left(1 - O(t^{-1})\right)  
\\ &=
%t (1 - \frac 12 \frac{\ln t}{t}  - 2\frac{\delta(t)}{t} + o(t^{-1}) )
%\left(1 - O(t^{-1})\right)  
%\\ &=
t - \frac 12 \ln t - 2 \delta(t) - O(1) 
\end{align*}
and shows 
$ - \ln (2\sqrt \pi) + o(1)=- 2 \delta(t) - O(1) $. 
Thus the function $z(t)$ given by (\ref{eq:def-z}) in fact satisfies (\ref{eq:et-Phi-eq}) and therefore it is the asymptotic expansion of the inverse function.

The desired estimate for $\d$ follows from some more elaborate but elementary calculations. First we prove that $\d \geq 0$. By (\ref{eq:et-Phi-eq2}) and because $e^x \geq 1+x$,
\begin{align*}
t - \frac 12 \ln t - \ln (2\sqrt \pi) + \frac 14 \frac{\ln t }{t} + \frac{\delta(t)}{ t} 
&\geq 
t \left(1 - \frac 12 \frac{\ln t}{t}  - 2\frac{\delta(t)}{t}  \right)
\left(1 - \frac{\theta}{2t} \right) ^2 
\\ &\geq 
 (t - \frac 12  \ln t   - 2 \delta(t) )
\left(1 -  \frac{\theta}{t} \right) 
\end{align*}
which is equivalent to 
\begin{align*}
\delta(t) 
&\geq 
\frac{ \ln (2\sqrt \pi) 
- \theta   - \frac{1-2\theta \ln t }{4t}   }
{(2 + \frac{1-2\theta  }{t} ) }
>0
\end{align*}
for $t \geq 1$.
On the other hand, again by (\ref{eq:et-Phi-eq2}),
\begin{align*}
t 
&\geq 
\left(t - \frac 12 \ln t - \ln (2\sqrt \pi) + \frac 14 \frac{\ln t }{t} + \frac{\delta(t)}{ t} \right)
e^{\frac 12 \frac{\ln t}{t}  + 2\frac{\delta(t)}{t} }
%\\ &\geq 
%(t - \frac 12 \ln t - \ln (2\sqrt \pi) + \frac 14 \frac{\ln t }{t} + \frac{\delta(t)}{ t} )
%(1+ \frac 12 \frac{\ln t}{t}  + 2\frac{\delta(t)}{t} )
\end{align*}
and using $e^x \geq 1+x$ implies
$$
\delta(t) 
\leq 
\frac{ \ln (2\sqrt \pi) 
+  \frac {2\ln (2\sqrt \pi) -1}4 \frac{\ln t}{t}  
+ \frac 14 \frac{(\ln t)^2}{t}   + \frac 18 \frac{(\ln t )^2}{t^2}   }
{2  - (2\ln (2\sqrt \pi)-1) \frac{1}{t}
- \frac{\ln t }{t} }
\leq 16
. $$
\iffalse
\begin{align*}
\delta(t) \left(2  - (2\ln (2\sqrt \pi)-1) \frac{1}{t}
 + \frac{ \ln t}{ t^2}   
 - \frac{\ln t }{t} + 2\frac{\delta(t)}{t^2} \right)
&\leq 
 \ln (2\sqrt \pi) 
+  \frac {2\ln (2\sqrt \pi) -1}4 \frac{\ln t}{t}  
+ \frac 14 \frac{(\ln t)^2}{t}   + \frac 18 \frac{(\ln t )^2}{t^2}  
\end{align*}

\begin{align*}
\delta(t) 
&\leq 
\frac{ \ln (2\sqrt \pi) 
+  \frac {2\ln (2\sqrt \pi) -1}4 \frac{\ln t}{t}  
+ \frac 14 \frac{(\ln t)^2}{t}   + \frac 18 \frac{(\ln t )^2}{t^2}   }
{2  - (2\ln (2\sqrt \pi)-1) \frac{1}{t}
- \frac{\ln t }{t} }
\leq 
\frac{ \ln (2\sqrt \pi) 
+  \frac {2\ln (2\sqrt \pi) -1}{4  e}
+  \frac 1{e^{2}}   + \frac 1{8 e^{2}}   }
{2  - (2\ln (2\sqrt \pi)-1) 1
- \frac 1{e} }
=15,41\dots
\end{align*}
\fi

\end{proof}

An asymptotic expansion for $\p (\Phi^{-1} (1-u))$ follows immediately:
\begin{lemma}\label{le:phi-Phi-inv}
For any $u \in (0,e^{-1}]$ there is a $\d=\d(u) \in (0,16)$ such that 
\begin{equation*}\label{eq:Phiinv-asymp2}
g_d(u)= \left( (2u)^{-1}\p(  \Phi^{-1}(1-u) ) \right)^{d-1}= 
(\ln u^{-1})^{  \frac {d-1}2 }
e^{- \frac {d-1}4 \frac{\lln u^{-1} }{ \ln u^{-1}} 
- (d-1)\frac{\delta}{\ln  u^{-1}} }~.
\end{equation*}
\end{lemma}

\section{Concentration of the $\beta$-distribution}
\label{sec:beta}

A basic integral for us is the Beta-integral
\begin{equation}
\label{beta}
\bB(\a,\b) =
\int\limits_{0}^{1}(1- x)^{\a-1} x^{\b -1}\,dx
=\frac{(\a-1)!(\b-1)!}{(\a+\b-1)!}.
\end{equation}
Let $U \sim \bB(\a, \b)$ distributed. Then
$\E U=\frac \b{\a+\b}$ and $\var(U)= \frac{\a \b}{(\a+\b)^2 (\a +\b+1)}$ 
Next we establish concentration inequalities for a Beta-distributed random variable 
around its mean. Observe that if $U \sim \bf B(\a,\b)$, then $1-U \sim \bf B (\b, \a)$. Hence we may concentrate on the case 
$\a\ge \b$.

\begin{lemma}
\label{le:betalow-}
Let $U \sim \bB(a+1, b+1)$ distributed with $a \geq b$ and set $n=a+b$. Then
$$
\P\left(U \leq  \frac bn - s \frac{a^{\frac 12} b^{\frac 12}}{n^{\frac 32}}\right) 
\leq
\frac {3e^3}{ \pi }
\frac 1{s}  \left( e^{- \frac 1{6} s^2  } 
- e^{- \frac 1{6} \frac {nb}a  } \right)_+
. $$

\end{lemma}

\begin{proof}
We have to estimate the integral
$$
\frac 1{\bB(a+1, b+1)} 
\int\limits_{0}^{\frac{b-s \sqrt{\frac{ab}n}}{n}}   (1- x)^a x^b \,dx
$$
For an estimate from above  
we substitute $x= \frac bn - \frac yn \sqrt{\frac{ab}n}$.
\begin{eqnarray*}
J_-
&=&
\int\limits_{0}^{\frac{b-s \sqrt{\frac{ab}n}}{n}}(1- x)^a x^b\,dx
\\ &=&
\frac { a^{a+ \frac 12} b^{b+ \frac 12} }{n^{n+\frac 32}} \  
\int\limits_{s }^{\sqrt{\frac {nb}{a}}}
\left(1+ y \sqrt{\frac{b}{an}}\right)^a 
\left(1 - y  \sqrt{\frac{a}{bn}}\right)^b\,dy
\end{eqnarray*}
It is well known that 
\begin{equation}\label{eq:est-ln}
\ln(1+x) = \sum_{k=1}^\infty (-1)^{k-1} \frac{x^k}k 
%\leq x - \frac{x^2}2 + \frac {x^3}3
\leq x - \frac{x^2}6 ,
\end{equation}
for $x \in (-1,1]$.
Since $a \geq b$, we have 
$$
\left(1+ y \sqrt{\frac{b}{an}}\right)^a 
\left(1 - y  \sqrt{\frac{a}{bn}}\right)^b
\leq
e^{- \frac {1}{6} y^2  }~,
$$
which implies
\begin{eqnarray*}
J_-
&\leq &
\frac { a^{a+ \frac 12} b^{b+ \frac 12} }{n^{n+\frac 32}} \  
\int\limits_{s }^{\sqrt{\frac {nb}{a}}}
e^{- \frac {1}{6} y^2  }
\,dy 
\\ &\leq &
\frac {3 a^{a+ \frac 12} b^{b+ \frac 12} }{n^{n+\frac 32}} \  
\frac 1{s}  \left( e^{- \frac 1{6} s^2  } 
- e^{- \frac 1{6} \frac {nb}a  } \right)~.  
\end{eqnarray*}
In the last step we use Stirling's formula,
$$
\sqrt{2\pi}\, n^{n+\frac 12} e^{-n}
\leq n! \leq 
e \, n^{n+\frac 12} e^{-n},
$$
to see that 
\begin{equation}\label{eq:stirling}
\frac { a^{a+ \frac 12} b^{b+ \frac 12} }{n^{n+\frac 32}} \  
\leq
\frac {e^3}{\pi }
\bB(a+1,b+1)
. 
\end{equation}
\end{proof}

\iffalse
The same considerations lead to the following lemma. Since this will not be needed in the 
following, we state it without proof.
\begin{lemma}
\label{betalow+}
Let $U \sim \bB(a+1, b+1)$ distributed with $a \geq b$  and set $n=a+b$. Then
$$
\P\left(U  \geq  \frac bn +  s \frac{a^{\frac 12} b^{\frac 12}}{n^{\frac 32}}\right) 
\leq
\frac {3e^3}{  \pi }
\frac 1{s}  \left( e^{- \frac 1{6} s^2  } 
- e^{- \frac 1{6} \frac {nb}a  } \right)_+
+
\P(U \geq 2 \frac bn)  .
$$

\end{lemma}
\fi

\begin{lemma}
\label{le:betalarge}
Let $U \sim \bB(a+1, b+1)$ distributed with $a \geq b$  and set $n=a+b$. Then
for $\l\geq 2$,
$$
\P(U \geq  \l \frac bn) 
\leq
\frac{e^3}{\pi} 
\l^b b^{\frac 12 } 
e^{b  + \frac 32}  
   e^{- \l \frac {ab}n} 
. 
 $$
\end{lemma}

\begin{proof}
We assume that $a \geq b$ and thus $a \geq \frac n2$.
We have to estimate the probability
\begin{eqnarray*}
\P(U \geq \l \frac bn) 
& \leq &
\frac 1{\bB(a+1, b+1)} 
\int\limits_{\l \frac bn}^{1}(1- x)^a x^b \,dx
\end{eqnarray*}
We substitute $x \to \frac 1a x + \l \frac bn$ and obtain
\begin{eqnarray*}
\int\limits_{\l \frac bn}^{1}(1- x)^a x^b \,dx
&\leq&
\int\limits_{0}^{\infty}e^{-x - \l \frac {ab}n} (\frac 1a x + \l 
\frac {b}n))^b \ \frac 1a \,dx
\\ &\leq&
a^{-(b+1)} e^{- \l \frac {ab}n} 
\int\limits_{0}^{\infty}e^{-x } (x + \l 
\frac {ab}n))^b \,dx.
\end{eqnarray*}
The use of the binomial formula and the Gamma functions yields
\begin{eqnarray*}
\int\limits_{0}^{\infty}e^{-x } (x + \l 
\frac {ab}n))^b \,dx
&=&
\sum_{k=0}^b {b \choose k}
\int\limits_{0}^{\infty}e^{-x } x^{b-k}  (\l 
\frac {ab}n)^k \,dx
\\ &=&
\sum_{k=0}^b {b \choose k}
(b-k)!  (\l 
\frac {ab}n)^k 
\\ &\leq&
b (\l  \frac {ab}n)^b
\end{eqnarray*}
because $ b \leq \l \frac{ab} n$ for $a \geq \frac n2 \geq b$ and $\l \geq 2$,  and 
$\frac {1}{k!} (\l  \frac {ab}n)^k 
$
is increasing for $ k\leq  (\l  \frac {ab}n) $.
Using (\ref{eq:stirling}) this gives 
\begin{eqnarray*}
\P(U \geq \l \frac bn) 
& \leq &
\frac {e^3}{\pi }
\left( 1+ \frac {b}{a} \right)^{a+ \frac 32}  
 b^{\frac 12 } \l^b  e^{- \l \frac {ab}n} 
\end{eqnarray*}
and with $(1+x) \leq e^x$ the lemma.
\end{proof}

\section{The case $n-d $ large}
\label{sec:onedim}

In this section we combine Lemma \ref{le:phi-Phi-inv} which gives the asymptotic behavior of
$  g_d(u) $ as $u \to 0$, with the concentration properties of the Beta function just obtained. We 
split our proof in two Lemmata.

\begin{lemma}\label{le:int-upper-bound}
For $d \geq d_0=78$  and $n \geq e^e d$ we have
$$
\E g_d(U) 
\leq
e^{  \frac {d-1}2 \lln (\frac nd ) 
- \frac{d-1}4 \frac{\lln (\frac {n}{d}) }{ \ln (\frac {n}{d})} 
+ (d-1) \frac{2}{ \ln (\frac nd)}  } 
e^{\frac {e^6}{\pi} \sqrt d e^{-\frac 1{10} d} } .
$$
\end{lemma}

\begin{lemma}\label{le:int-lower-bound}
For $d \geq d_0=78$  and $n \geq e^e d$ we have
$$
\E g_d(U) 
\geq 
e^{  \frac {d-1}2 \lln (\frac nd) 
- \frac{d-1}4 \frac {\lln \frac nd }{\ln \frac nd  }
-(d-1) \frac {34}{\ln \frac nd}    }
\  
e^{- \frac{2e^6}{\pi}  \sqrt d e^{ - \frac 1{10} d}} .
$$
\end{lemma}
These two bounds prove Theorem \ref{th:facetnlarge}.
The idea is to split the expectation into the main term close to $\frac dn$ and two error terms,
\begin{eqnarray*}
 \E g_d(U) 
& = &
\E g_d(U) \, \1\left( U \leq e^{-2} \frac dn \right)
\\&&
+\E g_d(U) \, \1 \left( U \in \Big[e^{-2} \frac dn,2 \frac dn\Big]\right)
\\&&
+\E g_d(U) \, \1 \left( U \geq 2 \frac dn \right)~.
\end{eqnarray*}

\begin{proof}[Proof of Lemma \ref{le:int-lower-bound}]
Recall that $U$ is $\bB(n-d+1,d)$-distributed.
Lemma \ref{le:betalarge} with $a=n -d $ and $b= d-1 $ shows that 
$$
\P\left(U \geq  \l \frac {d}{n}\right) 
\leq
\P\left(U \geq  \l \frac {d-1}{n-1}\right) 
\leq
\frac{e^3}{\pi} 
\l^{d-1} (d-1)^{\frac 12 } 
e^{(d-1)  + \frac 32}  
   e^{- \l \frac {(n-d)(d-1)}{n-1}} 
$$ 
because $ \frac {d-1}{n-1}< \frac dn$. For $\l=2$ this gives
\begin{equation}\label{eq:dev-2dn}
\P\left(U \geq  2 \frac dn\right) 
\leq
\frac{e^6}{2\pi}  \sqrt d
e^{ ( \ln 2- 1 +  2 \frac {d}n) d} 
\leq
\frac{e^6}{2\pi}  \sqrt d
e^{ - \frac 1{10} d}
\end{equation}
for $n \geq 10 d$. The probability that $U$ is small is estimated by Lemma \ref{le:betalow-} with 
$s= (1- e^{-2})\sqrt{\frac{(d-1)(n-1)}{n-d}}$,
\begin{eqnarray*}
\P \left( U \leq  e^{-2} \frac {d-1}{n-1} \right)
&\leq&
\frac {3e^3}{ \pi }
(1- e^{-2})^{-1} \sqrt{\frac{n-d}{(d-1)(n-1)}}  e^{- \frac 1{6} (1- e^{-2})^2 \frac{(d-1)(n-1)}{n-d}  } 
\\ &\leq&
\frac {e^6}{ 2 \pi }
 e^{- \frac 1{10}d } 
\end{eqnarray*}
for $d \geq 6$.
Combining both estimates and using 
\begin{equation}\label{eq:ln>2x}
\ln (1+x) \geq + 2x 
\end{equation}
for $x \in [0, \frac 12]$,
%$$(\ln (1-x) +2x)'= - \frac 1{1-x} +2 = 0$$
we have
\begin{equation}\label{eq:PUin}
\P \left( U \in \Big[\frac 12 \frac dn,2 
\frac dn \Big] \right)
\geq
1
- \frac{e^6}{2\pi}  \sqrt d e^{ - \frac 1{10} d}
- \frac {e^6}{ 2 \pi }  e^{- \frac 1{10}d } 
\geq
e^{ - \frac{2e^6}{\pi}  \sqrt d e^{ - \frac 1{10} d}} 
\end{equation}
for $d \geq d_0=78$. (Observe that 
$ 
\frac{2e^6}{\pi}  \sqrt d_0 e^{ - \frac 1{10} d_0}
\leq  
\frac 12 $.)
In the last step we compute
\begin{eqnarray*}
\min_{u \in [e^{-2} \frac dn,2 \frac dn]} g_d(u)
& = &
\min_{u \in [e^{-2} \frac dn,2 \frac dn]}
e^{\frac{d-1}2 \lln u^{-1} 
- \frac{d-1}4 \frac{\ln \ln u^{-1} }{ \ln u^{-1}} - (d-1) \frac{ \delta}{ \ln u^{-1}} }
\\ &\geq&
e^{  \frac {d-1}2 \lln (\frac 12 \frac nd)  
- \frac{d-1}4 \frac{\lln (\frac 12 \frac {n}d) }{ \ln (\frac 12 \frac n{d})} 
- (d-1) \frac{\max \delta}{ \ln (\frac 12 \frac n{d})} } 
\end{eqnarray*}
for $n \geq e^e d$.
Here, note that $\frac {\lln x}{\ln x}$ is decreasing for $x \geq e^e$.
Now using
$$
\lln \left(\frac nd \right)  
\geq 
\lln \left(\frac 12 \frac nd \right)  
 =
\lln \left(\frac nd \right)  +\ln \left(1- \frac{\ln 2 }{ \ln (\frac nd)} \right)
\geq 
\lln \left(\frac nd \right) - \frac{2\ln 2 }{ \ln (\frac nd)},
$$
and 
$$
\frac 1{\ln (\frac 12 \frac nd)}  
 =
\frac 1{\ln (\frac nd)  - \ln 2 }
\leq
\frac 1{\ln (\frac nd ) }
\left( 1 + 2 \frac{\ln 2}{\ln (\frac nd) } \right)
\leq
2  \frac 1{\ln (\frac nd)  }
$$
for $n \geq e^{e} d $, we have
\iffalse
\begin{eqnarray*}
\min_{u \in [e^{-2} \frac dn,2 \frac dn]} g_d(u)
&\geq&
e^{  \frac {d-1}2(\lln \frac nd - \frac{2\ln 2 }{ \ln \frac nd})
- \frac{d-1}4 (\lln \frac nd ) \frac 1{\ln \frac nd  }
\left( 1 +2 \frac{\ln 2}{\ln \frac nd } \right)
- (d-1) \max \delta\frac 2{\ln \frac nd  }
}
\\&\geq&
e^{  \frac {d-1}2 \lln \frac nd 
- \frac{d-1}4 \frac {\lln \frac nd }{\ln \frac nd  }
-(d-1) \frac 1{\ln \frac nd} [ 
+  \ln 2 
+ \frac{2 \ln 2}4 \frac{\lln \frac nd  }{\ln \frac nd } 
+  \max 2 \delta ]
} .
\\&\geq&
e^{  \frac {d-1}2 \lln \frac nd 
- \frac{d-1}4 \frac {\lln \frac nd }{\ln \frac nd  }
-(d-1) \frac 1{\ln \frac nd} [ 
+  \ln 2 
+ \frac{\ln 2}2  
+  \max 2 \delta ]
} .
\\&\geq&
e^{  \frac {d-1}2 \lln \frac nd 
- \frac{d-1}4 \frac {\lln \frac nd }{\ln \frac nd  }
-(d-1) \frac 1{\ln \frac nd} [ 
 \frac{3 \ln 2}2
+  \max 2 \delta ]
} .
\end{eqnarray*}
\fi
\begin{eqnarray*}
\min_{u \in [e^{-2} \frac dn,2 \frac dn]} g_d(u)
&\geq&
e^{  \frac {d-1}2 \lln \frac nd 
- \frac{d-1}4 \frac {\lln \frac nd }{\ln \frac nd  }
-(d-1) \frac {\delta '}{\ln \frac nd} 
} 
\end{eqnarray*}
with $\delta' =   \frac{3 \ln 2}2  + 2 \max \delta \in [0, 34]$. Combinig this estimate with 
(\ref{eq:PUin}) we obtain
\begin{eqnarray*}
\E g_d(U) 
&\geq&
\min_{u \in [e^{-2} \frac dn,2 \frac dn]} g_d(u)\   \E  \1\left( U \in \Big[e^{-2} \frac dn,2 
\frac dn\Big] \right)
\\ & \geq &
e^{  \frac {d-1}2 \lln \frac nd 
- \frac{d-1}4 \frac {\lln \frac nd }{\ln \frac nd  }
-(d-1) \frac {\delta '}{\ln \frac nd}   } \  
e^{- \frac{2e^6}{\pi}  \sqrt d e^{ - \frac 1{10} d}} 
\end{eqnarray*}
for $d \geq d_0$ and $n \geq e^e d$.
\end{proof}

\begin{proof}[Proof of Lemma \ref{le:int-upper-bound}]
As an upper bound we have
\begin{eqnarray*}
\E g_d(U) 
&\leq&
\E g_d(U)\1\left( U \leq e^{-2} \frac dn \right)
\\ && +
\max_{u \in [e^{-2} \frac dn,2 \frac dn]} g_d(u)
\   \P \left( U \in \Big[e^{-2} \frac dn,2  \frac dn \Big] \right)
\\&&
+ \underbrace{ \max_{u \in [2 \frac dn,1]} g_d(u)}_{\leq \max_{u \in [\frac dn,1]} g_d(u)}
\P \left( U \geq 2 \frac dn \right) 
%----------------------------------------------------
\\ &\leq&
\E g_d(U)\1\left( U \leq e^{-2} \frac dn \right)
\\ && +
e^{  \frac {d-1}2 \lln (e^2 \frac {n}{d})  - \frac{d-1}4 \frac{\lln (e^2 \frac {n}{d}) }{ \ln 
(e^2 \frac {n}{d})} }
\\&&
+ e^{  \frac {d-1}2 \lln (\frac n{d})  - \frac{d-1}4 \frac{\lln (\frac {n}{d}) }{ \ln 
(\frac n{d})} }
\frac{e^6}{2\pi}  \sqrt d 
e^{ - \frac 1{10} d} 
\end{eqnarray*}
since $\delta \geq 0$, and 
where the last term follows from (\ref{eq:dev-2dn}).
\iffalse
$$
\P(U \geq  \l \frac dn) 
\leq
\frac{e^2}{2\pi}  \sqrt d 
e^{ d + \frac 32 + d \ln \l- \l \frac {d(n-d)}n} 
\leq
\frac{e^4}{2\pi}  \sqrt d 
e^{ d (1+ \ln \l- \l + \l \frac {d}n)} 
 $$
Because $\frac{1 + \ln \l }\l$ is decreasing and thus bounded from above by $\frac{1 + \ln 2 }2$,
$$ 1 + \ln \l - \l + \l \frac dn \leq - \frac 1 {24} \l$$
%$$ \frac{1 + \ln \l }\l \leq \frac{1 + \ln 2 }2 \leq  (\frac {23} {24}- \frac 1{10} ) \leq (\frac {23} {24}- \frac dn )$$  
for $\l \geq 2$ and $n \geq 10 d$. 
This implies
$$
\P(U \geq  \l \frac dn) 
\leq
\frac{e^4}{2\pi}  \sqrt d 
e^{ - \frac 1{24}\l d} 
 $$
\fi 
For the first term we use that $\p(\Phi^{-1}(\cdot))$ is a symmetric and concave function and thus increasing on $[0, e^{-2}\frac d{ n}]$, and that $\d \geq 0$.
\begin{eqnarray*}
\E \lefteqn{g_d(U)\1 \left( U \leq e^{-2}\frac d{n} \right) } 
&&\\ &\leq &
\frac 1{\bB(n-d+1, d)}
\int\limits_0^{e^{-2}\frac d{n}} 
 e^{\frac {d-1}2 \lln x^{-1} - \frac{d-1}4 \frac{\lln x^{-1} }{ \ln x^{-1}}   }
(1-x)^{n-d} x^{d-1} dx
%---------------------------------------
\\ & \leq &
\frac 1{\bB(n-d+1, d)}
e^{  \frac {d-1}2 \lln (e^2 \frac {n}{d})  - \frac{d-1}4 \frac{\lln (e^2 \frac {n}{d}) }{ \ln 
(e^2 \frac {n}{d})} }
\left(e^{-2} \frac d{n}\right)^{d-1} 
\int\limits_0^{\infty} 
e^{-(n-d)x} dx
\end{eqnarray*}
Now the remaining integration is trivial. We use Stirling's formula (\ref{eq:stirling}) to estimate the Beta-function and obtain
\begin{eqnarray*}
\E \lefteqn{g_d(U)\1 \left( U \leq e^{-2}\frac d{n} \right) } 
&&\\ &\leq &
\frac {e^3}{\pi} \frac {(n-1)^{n +\frac 12 }}{(n-d)^{n-d+\frac 32} (d-1)^{d- \frac 12}}
e^{  \frac {d-1}2 \lln (e^2 \frac {n}{d})  - \frac{d-1}4 \frac{\lln (e^2 \frac {n}{d}) }{ \ln 
(e^2 \frac {n}{d})} }
\left(e^{-2} \frac d{n}\right)^{d-1} 
%---------------------------------------
\\ & \leq &
e^{  \frac {d-1}2 \lln (e^2 \frac {n}{d})  - \frac{d-1}4 \frac{\lln (e^2 \frac {n}{d}) }{ \ln 
(e^2 \frac {n}{d})} }
\frac {e^5}{\pi} 
e^{(d-1) +
\frac {(d-1)}{(n-d) }(\frac 32) +
1+ \frac {1}{(d-1)}\frac 12  
-2d} 
%---------------------------------------
%\\ & \leq &
%e^{  \frac {d-1}2 \lln (e^2 \frac {n}{d})  
%- \frac{d-1}4 \frac{\lln (e^2 \frac {n}{d}) }{ \ln 
%(e^2 \frac {n}{d})} }
%\frac {e^5}{\pi} 
%e^{-d( 1- \frac {1}{(n-d) }(\frac 32) - \frac {1}{d(d-1)}\frac 12  )
%} 
%---------------------------------------
\\ & \leq &
e^{  \frac {d-1}2 \lln (e^2 \frac {n}{d})  - \frac{d-1}4 \frac{\lln (e^2 \frac {n}{d}) }{ \ln 
(e^2 \frac {n}{d})} }
\frac {e^5}{\pi} 
e^{-\frac 1{10} d} 
\end{eqnarray*}
e.g. for $n \geq e^e d$ and $d \geq 78$. Combining our results gives
\begin{eqnarray*}
\E g_d(U) 
&\leq&
e^{  \frac {d-1}2 \lln (e^2 \frac {n}{d})  - \frac{d-1}4 \frac{\lln (e^2 \frac {n}{d}) }{ \ln 
(e^2 \frac {n}{d})} }
\frac {e^5}{\pi} 
e^{-\frac 1{10} d} 
\\ && +
e^{  \frac {d-1}2 \lln (e^2 \frac {n}{d})  - \frac{d-1}4 \frac{\lln (e^2 \frac {n}{d}) }{ \ln 
(e^2 \frac {n}{d})} }
\\&&
+ e^{  \frac {d-1}2 \lln (\frac n{d})  - \frac{d-1}4 \frac{\lln (\frac {n}{d}) }{ \ln 
(\frac n{d})} }
\frac{e^6}{2\pi}  \sqrt d 
e^{ - \frac 1{10} d} 
\end{eqnarray*}
In a similar way as above, we get rid of the involved constant $e^2$ by using
$$
\lln \left(\frac nd \right)  
\leq 
\lln \left(e^2 \frac nd \right)  
 =
\lln \left(\frac nd \right)  +\ln \left(1 + \frac{ 2 }{ \ln (\frac nd)} \right)
\leq 
\lln \left(\frac nd \right) + \frac{2}{ \ln (\frac nd)},
$$
and 
$$
\frac 1{\ln (e^2 \frac nd)}  
 =
\frac 1{\ln (\frac nd)}
\left( 1+ \frac 2{\ln (\frac nd)} \right)^{-1}
\geq
\frac 1{\ln (\frac nd ) }
\left( 1 -  \frac{2}{\ln (\frac nd) } \right) .
$$
This yields
\iffalse 
\begin{eqnarray*}
\E g_d(U) 
&\leq&
e^{  \frac {d-1}2 \lln (e^2 \frac {n}{d})  - \frac{d-1}4 \frac{\lln (e^2 \frac {n}{d}) }{ \ln 
(e^2 \frac {n}{d})} }
\left(1+ \frac {e^5}{\pi} e^{-\frac 1{10} d} \right)
\\&&
+ e^{  \frac {d-1}2 \lln (\frac n{d})  - \frac{d-1}4 \frac{\lln (\frac {n}{d}) }{ \ln 
(\frac n{d})} }
\frac{e^6}{2\pi}  \sqrt d 
e^{ - \frac 1{10} d} 
%--------------------------------
\\&\leq&
e^{  \frac {d-1}2 (\lln (\frac nd ) + \frac{2}{ \ln (\frac nd)})  
- \frac{d-1}4 \frac{\lln (\frac {n}{d}) }{ \ln (\frac {n}{d})} ( 1 -  \frac{2}{\ln (\frac nd) } 
) } 
\left(1+ \frac {e^5}{\pi} e^{-\frac 1{10} d} \right)
\\&&
+ e^{  \frac {d-1}2 \lln (\frac n{d})  - \frac{d-1}4 \frac{\lln (\frac {n}{d}) }{ \ln 
(\frac n{d})} }
\frac{e^6}{2\pi}  \sqrt d 
e^{ - \frac 1{10} d} 
%--------------------------------
\\&\leq&
e^{  \frac {d-1}2 \lln (\frac nd ) 
- \frac{d-1}4 \frac{\lln (\frac {n}{d}) }{ \ln (\frac {n}{d})} 
+ (d-1) \frac{1}{ \ln (\frac nd)} [1    
+ \frac{1}2 \frac{\lln (\frac {n}{d}) }{ \ln (\frac {n}{d})}] } 
\left(1+ \frac {e^6}{2\pi} \sqrt d e^{-\frac 1{10} d} \right)
\\&&
+ e^{  \frac {d-1}2 \lln (\frac n{d})  - \frac{d-1}4 \frac{\lln (\frac {n}{d}) }{ \ln 
(\frac n{d})} }
\frac{e^6}{2\pi}  \sqrt d 
e^{ - \frac 1{10} d} 
%--------------------------------
\\&\leq&
e^{  \frac {d-1}2 \lln (\frac nd ) 
- \frac{d-1}4 \frac{\lln (\frac {n}{d}) }{ \ln (\frac {n}{d})} 
+ (d-1) \frac{\frac{3}2}{ \ln (\frac nd)}  } 
\left(1+ \frac {e^6}{\pi} \sqrt d e^{-\frac 1{10} d} \right)
\end{eqnarray*}
\fi
\begin{equation} \label{eq:Ef-upperbound}
\E g_d(U) 
\leq
e^{  \frac {d-1}2 \lln (\frac nd ) 
- \frac{d-1}4 \frac{\lln (\frac {n}{d}) }{ \ln (\frac {n}{d})} 
+ (d-1) \frac{\frac{3}2}{ \ln (\frac nd)}  } 
\left(1+ \frac {e^6}{\pi} \sqrt d e^{-\frac 1{10} d} \right)
\end{equation}

\end{proof}

\section{The case $n-d$ small}
\label{sec:small}

Finally, it remains to prove Theorem \ref{th:facetn-dsmall}.
The starting point here is again formula (\ref{eq:main-integral}), together with the substitution $y \to 
\frac y{\sqrt d}$.
\begin{eqnarray}\label{eq:int-n-dsmall}
\E f_{d-1}([X_1,\ldots,X_n])
&=& \nonumber
2 {n \choose d}
\frac{\sqrt{d}}{\sqrt \pi}
\int\limits_{- \infty}^\infty \Phi(y)^{n-d} e^{-dy^2}\,dy
\\ &=&
2 {n \choose d}
\frac{1}{\sqrt \pi}
\int\limits_{- \infty}^\infty \Phi\Big(\frac y{\sqrt d}\Big)^{n-d} e^{-y^2}\,dy
\end{eqnarray}
The Taylor expansion of $ \Phi$ at $y=0$ is given by
\begin{eqnarray*}
\Phi(y) 
&=& 
\frac 12 + \frac{1}{\sqrt \pi} y + \frac 1{\sqrt \pi}  (-\theta_1) e^{-\theta_1^2} \,  y^2
= 
\frac 12 + \frac{1}{\sqrt \pi} y (1-\theta_2 y)
\end{eqnarray*}
with some 
$ \theta_1, \theta_2 \in \R $ depending on $y$. Since $\Phi(y)$ is
above its tangent at $0$ for $y>0$ and below it for $y<0$, we have $0 \leq 1- \theta_2 y \leq 1$. 
Further, 
$$ |\theta_2| \leq \max_{\theta_1} \theta_1 e^{-\theta_1^2} = \frac 1{\sqrt{2e}} . $$
Hence an expression for $\ln \Phi$ at $y=0$ is given by
\begin{eqnarray*}
\ln \Phi(y) 
&=& 
- \ln 2 + 
\ln \left( 1 + \frac{2}{\sqrt \pi} y (1- \theta_2 y) \right) .
\end{eqnarray*}
We need again estimates for the logarithm, namely  $\ln (1+x) = x - \theta_3 x^2 < x$ with some $\theta_3 = \theta_3 (x) \geq 0$. 
In addition, there exists  $c_3 \in \R$ such that  $\theta_3 < c_3 $
if $x$ is bounded away from $-1$, for example, for  $x \geq 2 \Phi(-1)-1$.
This gives
\begin{eqnarray*}
\ln \Phi(y) 
&\leq&
- \ln 2 
+  \frac{2}{\sqrt \pi} y 
-  \frac{2}{\sqrt \pi} \theta_2 y^2 
\end{eqnarray*}
and
\begin{eqnarray*}
\ln \Phi(y) 
&=& 
- \ln 2 
+  \frac{2}{\sqrt \pi} y (1- \theta_2 y) 
- \theta_3 \frac{4}{\pi} y^2 \underbrace{(1- \theta_2 y)^2}_{\leq 1}
\\ & \geq &
- \ln 2 
+  \frac{2}{\sqrt \pi} y  
-  \frac{2}{\sqrt \pi} \theta_2 y^2 
- \theta_3 \frac{4}{\pi} y^2
\end{eqnarray*}
with $\theta_3 < c_3 $ for  $y \geq -1$.
Thus the Taylor expansion of $\ln \Phi$ at $y=0$ is given by
\begin{eqnarray*}
\ln \Phi(y) 
&=& 
- \ln 2 + 
\frac{2}{\sqrt \pi} y - \theta_4 y^2
\end{eqnarray*}
with some 
$ \theta_4= \theta_4(y) > - \frac 12$, and there exists a $c_4 \in \R$ with $\theta_4 \leq c_4$ for $y \geq -1$.
We plug this into (\ref{eq:int-n-dsmall}) and obtain
\begin{eqnarray*}
\int\limits_{- \infty}^\infty \Phi\Big(\frac y{\sqrt d}\Big)^{n-d} e^{-y^2}\,dy
& = &
e^{-  (n-d)\ln 2} 
\int\limits_{- \infty}^\infty 
e^{ \frac{2}{\sqrt \pi} \frac{n-d}{\sqrt d} y - \theta_4 \frac{n-d}{d} y^2 -y^2 } \,dy~.
\end{eqnarray*}
Since $\frac {n-d}d \to 0$ we assume that 
$ 1+\theta_4 \frac{n-d}{d} \geq 1- \frac 12  \frac{n-d}{d} > 0.$
As an estimate from above we have
\begin{eqnarray}
\int\limits_{- \infty}^\infty 
e^{ \frac{2}{\sqrt \pi} \frac{n-d}{\sqrt d} y - (1+\theta_4 \frac{n-d}{d}) y^2 } \,dy
& \leq & \nonumber
\int\limits_{- \infty}^\infty 
e^{ \frac{2}{\sqrt \pi} \frac{n-d}{\sqrt d} y - (1- \frac 12  \frac{n-d}{d}) y^2 } \,dy
\\ & = & \nonumber
e^{\frac {\frac 4\pi \frac{(n-d)^2}{d}}{4 (1- \frac 12 \frac {n-d}d)}}
\int\limits_{- \infty}^\infty 
e^{ -\left( \frac{\frac{2}{\sqrt \pi} \frac{n-d}{\sqrt d}}{2 \sqrt{(1- \frac 12  \frac{n-d}{d})}} - \sqrt{(1- \frac 12  \frac{n-d}{d})}y\right)^2 } \,dy
\\ & = & \nonumber
e^{ \frac 1\pi \frac{(n-d)^2}{d} (1+ O(\frac {n-d}d))}
\frac {\sqrt \pi}{\sqrt{(1- \frac 12  \frac{n-d}{d})} }
\\ & = & \label{eq:small-est-1}
\sqrt \pi e^{ \frac 1\pi \frac{(n-d)^2}{d} + O(\frac {(n-d)^3}{d^2}) + O(\frac{n-d}{d}) } .
\end{eqnarray}
The estimate from below is slightly more complicated. For $y \geq -\sqrt d$ there is an upper bound $c_4$ for $\theta_4$. Using this we have
\begin{eqnarray*}
\int\limits_{- \infty}^\infty 
e^{ \frac{2}{\sqrt \pi} \frac{n-d}{\sqrt d} y - \theta_4 \frac{n-d}{d} y^2 -y^2 } \,dy
& \geq  & \nonumber
e^{\frac 1\pi \frac{(n-d)^2}d}
\int\limits_{\frac 1{\sqrt \pi} \frac{n-d}{\sqrt d}- \sqrt d}^\infty 
e^{-\left( \frac{1}{\sqrt \pi} \frac{n-d}{\sqrt d} - y \right)^2 - c_4 \frac{n-d}{d} y^2  } \,dy
\\& \geq  & \nonumber
e^{\frac 1\pi \frac{(n-d)^2}d}
\int\limits_{- \infty}^{\sqrt d }   
e^{-y^2 - c_4 \frac{n-d}{d} \left( \frac{1}{\sqrt \pi} \frac{n-d}{\sqrt d} - y \right)^2  } \,dy~.
%\\& \geq &
%e^{\frac 1\pi \frac{(n-d)^2}d}
%\int\limits_{- \infty}^{\sqrt d}   
%e^{-y^2 - 2 c_4 \frac{n-d}{d} \left( \frac{1}{\pi} \frac{(n-d)^2}{d} + y^2 \right)  } \,dy
\end{eqnarray*}
Now we use $(a-b)^2 \leq 2a^2 + 2b^2$ which shows that 
\begin{eqnarray}
\int\limits_{- \infty}^\infty 
e^{ \frac{2}{\sqrt \pi} \frac{n-d}{\sqrt d} y - \theta_4 \frac{n-d}{d} y^2 -y^2 } \,dy
& \geq & \nonumber
e^{\frac 1\pi \frac{(n-d)^2}d + O(\frac{(n-d)^3}{d^2}) }
\int\limits_{- \infty}^{\sqrt d }   
e^{- (1+ 2 c_4 \frac{n-d}{d}) y^2 } \,dy
\\ & = & \nonumber
e^{\frac 1\pi \frac{(n-d)^2}d + O(\frac{(n-d)^3}{d^2}) }
\frac 1 {\sqrt{(1+ 2 c_4 \frac{n-d}{d})}}
\int\limits_{- \infty}^{\sqrt {d (1+ 2 c_4 \frac{n-d}{d})}}   
e^{- y^2 } \,dy
\\ & \geq & \label{eq:small-est-2}
e^{\frac 1\pi \frac{(n-d)^2}d + O(\frac{(n-d)^3}{d^2}) +O(\frac{n-d}{d})}
\int\limits_{- \infty}^{\sqrt d }   
e^{- y^2 } \,dy .
\end{eqnarray}
Recall the estimate for $\Phi (z) $ from Lemma \ref{le:Phi-asymp}, 
\begin{equation}\label{eq:small-est-3}
\int\limits_{- \infty}^{\sqrt d }   
e^{- y^2 } \,dy 
= 
\sqrt \pi \, \Phi(\sqrt d)
\geq 
\sqrt \pi (1- e^{-d})
=
\sqrt \pi e^{O(e^{-d}) }~.
\end{equation}
We combine equations \eqref{eq:small-est-1}, \eqref{eq:small-est-2} and \eqref{eq:small-est-3} and obtain
$$
\int\limits_{- \infty}^\infty 
e^{ \frac{2}{\sqrt \pi} \frac{n-d}{\sqrt d} y - \theta_4 \frac{n-d}{d} y^2 -y^2 } \,dy
= 
\sqrt \pi  e^{\frac 1\pi \frac{(n-d)^2}d + O(\frac{(n-d)^3}{d^2}) +O(\frac{n-d}{d}) + O(e^{-d}) }
$$
which yields Theorem \ref{th:facetn-dsmall}.

\medskip
\noindent {\bf Acknowledgement: } We would like to thank Imre B\'ar\'any and Gerg\H{o} Nemes for 
fruitful discussions.

\end{document}